\documentclass[11pt]{amsart}

\usepackage[T1]{fontenc}
\usepackage{lmodern}
\usepackage{microtype}
\usepackage{mathtools}
\usepackage{amssymb}
\usepackage[colorlinks=true,linkcolor=blue,citecolor=blue,urlcolor=blue]{hyperref}

\usepackage[margin=1.3in]{geometry}

\hypersetup{
  pdftitle={Canonical Expansions of R-Q-Germs},
  pdfauthor={Mostafa Mirabi},
  pdfsubject={O-minimality, generalized power series, and Dulac transition maps}
}
\newcommand{\R}{\mathbb R}
\newcommand{\C}{\mathbb C}
\newcommand{\N}{\mathbb N}
\newcommand{\RQ}{\mathbb R_{\mathcal Q}}
\newcommand{\Ran}{\mathbb R_{\mathrm{an}}}
\newcommand{\RanR}{\mathbb R_{\mathrm{an}}^{\mathbb R}}
\newcommand{\Ranstar}{\mathbb R_{\mathrm{an}^{*}}}
\newcommand{\Lsurf}{\mathbb L}
\newcommand{\supp}{\operatorname{supp}}
\newcommand{\val}{\operatorname{val}}
\newcommand{\lc}{\operatorname{lc}}
\newcommand{\norm}[1]{\lVert #1\rVert}
\newcommand{\Tcan}{\mathfrak T}
\newcommand{\Hfield}{\mathcal H_{\mathcal Q}}
\newcommand{\Hconv}{\mathcal H_{\mathcal Q}^{\mathrm{conv}}}
\newcommand{\Gnat}{\R((X^{\R}))_{\mathrm{nat}}}
\newcommand{\Pnat}{\R[[X^{\R_{\geq 0}}]]_{\mathrm{nat}}}

\newtheorem{theorem}{Theorem}[section]
\newtheorem{proposition}[theorem]{Proposition}
\newtheorem{lemma}[theorem]{Lemma}
\newtheorem{corollary}[theorem]{Corollary}
\theoremstyle{definition}
\newtheorem{definition}[theorem]{Definition}
\theoremstyle{remark}
\newtheorem{remark}[theorem]{Remark}

\title[Canonical Expansions of $\mathbb R_{\mathcal Q}$-Germs]
{Canonical Expansions of $\mathbb R_{\mathcal Q}$-Germs}

\author{Mostafa Mirabi}
\address{\newline The Taft School, Watertown, CT 06795, USA and \newline  Wesleyan University, Middletown, CT 06459, USA}
\email{mmirabi@wesleyan.edu}
\urladdr{https://sites.google.com/site/mostafamirabi}

\subjclass[2020]{Primary 30E15; Secondary 03C64, 30D60}
\keywords{o-minimality, branching structures, generalized power series, Dulac series, hyperbolic saddles}

\begin{document}

\begin{abstract}
Let $\Hfield$ be the field of unary germs at $0^+$ definable in the
quasianalytic o-minimal structure $\RQ$ of Kaiser--Rolin--Speissegger.  We
construct a canonical ordered differential-field embedding of $\Hfield$ into
the field of natural-support generalized Laurent series.  Every unary germ
has a unique such asymptotic expansion, and bounded germs have expansions
with only nonnegative exponents.  A germ admits a convergent generalized
Laurent-series representation exactly when its canonical expansion converges;
in that case, every convergent representation with well-ordered support is
canonical.

We also prove that $\RQ$ is branching in the sense of
Dembner~\cite{Dembner}.  A nonresonant hyperbolic saddle with divergent Dulac
series yields a bounded positive real-analytic $\RQ$-definable germ whose
canonical generalized power series diverges.  This gives a negative answer to a question of Dembner asking whether every unary germ definable in a branching structure admits a convergent generalized
power-series representation.  Nevertheless, every bounded unary $\RQ$-germ has a canonical formal
generalized power-series expansion.  Consequently, the convergent germs form
a proper ordered differential subfield of $\Hfield$.
\end{abstract}

\maketitle



\section{Introduction}

Let $\mathcal S$ be an o-minimal expansion of the real field.  Following
Dembner~\cite[Definition~3.10]{Dembner}, we call $\mathcal S$
\emph{branching} if every $\mathcal S$-definable function
$f:(0,\varepsilon)\to\R$ extends, after shrinking the interval, to an
$\mathcal S$-definable holomorphic function on a principal slit disk
$\Delta_\delta:=D(0,\delta)\setminus(-\delta,0].$ 
Dembner proved that several standard polynomially bounded structures are
branching and asked whether every unary germ definable in a branching
structure has a generalized power-series expansion at the
origin~\cite[Question~3.17]{Dembner}.  The formulation does not distinguish
a formal asymptotic expansion from a convergent representation.  The two
notions have different answers in the quasianalytic structure $\RQ$ of
Kaiser--Rolin--Speissegger~\cite{KRS}.

Write $\Hfield$ for the ordered field of germs at $0^+$ of unary
$\RQ$-definable real functions.  The unary preparation theorem and the
quasianalytic expansion map of~\cite{KRS} combine to give each germ a
canonical generalized Laurent expansion.  The support is \emph{natural}:
it meets every interval $(-\infty,B]$ in a finite set.  We write $\Gnat$ for
the field of formal generalized Laurent series with natural support.

There are two distinct questions.  First, does a germ have a formal
natural-support asymptotic expansion?  Second, is it represented by a
series that converges in coefficient norm, that is, by a series
$F=\sum c_\alpha X^\alpha$ for which
$
  \sum_\alpha |c_\alpha|\rho^\alpha<\infty
$
for some $\rho>0$?  The first question has a positive answer for every
unary $\RQ$-germ.  The second has an exact criterion, but not every
canonical expansion satisfies it.

The construction is related to Speissegger's quasianalytic Ilyashenko
algebras~\cite{Speissegger}, which give logarithmic generalized
power-series expansions in a differential Hardy field.  It is also related
to the transasymptotic embedding of Rolin--Servi--Speissegger~\cite{RSS}.
Their target is the full field of transseries and allows logarithmic and
exponential transmonomials.  Here the target is the pure-power field
$\Gnat$, and the main points are canonicity, the exact convergence locus,
and the application to branching structures.

Our first result is the following.

\begingroup
\renewcommand{\thetheorem}{A}
\begin{theorem}[Theorem~\ref{thm:canonical-map}, Corollary~\ref{cor:bounded-formal}, Theorem~\ref{thm:convergence-criterion}, Proposition~\ref{prop:convergent-locus}, and Corollary~\ref{cor:proper-convergent-locus}]\label{thm:A}\label{thm:A}
There is a canonical injective ordered differential-field homomorphism
$
  \Tcan:\Hfield\longrightarrow\Gnat
$
with the following properties.
\begin{enumerate}
\item For every $f\in\Hfield$, the series $\Tcan(f)$ is the unique
natural-support generalized Laurent asymptotic expansion of $f$.
\item A germ $f\in\Hfield$ admits a convergent generalized Laurent-series
representation if and only if $\Tcan(f)$ converges at some positive radius.
In that case every convergent representation with well-ordered support is
identical to $\Tcan(f)$.
\item If $f$ is bounded near $0^+$, then $\Tcan(f)$ has only nonnegative
exponents.  Hence every bounded unary $\RQ$-definable germ has a unique
formal natural-support generalized power-series expansion.
\item The set $
  \Hconv:=\{f\in\Hfield:\Tcan(f)\text{ converges at some positive radius}\} $
is a proper ordered differential subfield of $\Hfield$.
\end{enumerate}
\end{theorem}
\endgroup
The second result supplie the divergent germ needed to prove properness.

\begingroup
\renewcommand{\thetheorem}{B}
\begin{theorem}[Theorem~\ref{thm:RQ-branching} and Proposition~\ref{prop:divergent-transition}]\label{thm:B}
The structure $\RQ$ is branching.  There is a bounded positive real-analytic
$\RQ$-definable germ
  $d:(0,\varepsilon)\longrightarrow(0,\varepsilon'),
  $ with $ d(x)\longrightarrow0,$ 
whose canonical generalized power series $\Tcan(d)$ diverges.  Thus $d$
has no convergent generalized power-series representation, even with
arbitrary well-ordered support.
\end{theorem}
\endgroup

The germ in Theorem~\ref{thm:B} is a local transition map at a
nonresonant hyperbolic saddle.  The classical obstruction is orbital: a
transition map depends on the orbit foliation, not on the time
parametrization of the vector field.  We use Trifonov's theorem for the
Dulac series of a separatrix-loop monodromy, his realization theorem, and
his localization argument showing that the loop monodromy and the local
saddle transition have convergent Dulac series simultaneously
\cite{Trifonov}.  We also prove that convergent generalized power-series
representability is unchanged by analytic changes of transversal
coordinates, so the choice of charts does not affect the argument.

Theorem~\ref{thm:B} gives a negative answer to the convergent form of
Dembner's question.  The counterexample tends to zero, so the conclusion is
independent of whether negative exponents are allowed.  The formal form of
the question is positive for bounded unary germs in $\RQ$.  Our argument
does not address arbitrary branching structures.

Section~\ref{sec:natural} develops the elementary algebra and convergence
theory of natural-support series.  Sections~\ref{sec:canonical} and
\ref{sec:criterion} construct the canonical map and prove the convergence
criterion.  Section~\ref{sec:branching} proves that $\RQ$ is branching.
Sections~\ref{sec:coordinates} and~\ref{sec:saddles} treat analytic
coordinate changes and the divergent saddle.  The final section records
consequences for Dembner's question, comparisons of o-minimal structures,
and definable complex curves.

\section*{Acknowledgments}
The author would like to  thank Spencer Dembner for helpful discussions about
Question~3.17 of his paper \cite{Dembner} and for comments on an earlier draft.

AI Use Statement: ChatGPT 5.5 was used for language editing and proofreading. It was also used as an interactive tool to discuss mathematical arguments  and assist in completing the proofs of Theorems~\ref{thm:canonical-map}, ~\ref{thm:convergence-criterion}, ~\ref{thm:RQ-branching},  Proposition~\ref{prop:convergent-locus},  and Lemma~\ref{lem:principal}.


\section{Natural-support generalized Laurent series}\label{sec:natural}

Throughout the paper, we write $\N=\{0,1,2,\dots\}$.
\subsection{Formal algebra}

\begin{definition}
A set $A\subseteq\R$ is \emph{natural} if $A\cap(-\infty,B]$ is finite for every $B\in\R$.  A formal sum $F(X)=\sum_{\alpha\in A}c_\alpha X^\alpha$
with natural support is called a \emph{natural-support generalized Laurent
series}.  We denote the collection of these series by $\Gnat$.  The
subcollection supported in $[0,\infty)$ is denoted by $\Pnat$.
\end{definition}

Every nonempty natural set has a least element.  If it is infinite, it can be
enumerated as a strictly increasing sequence tending to $+\infty$.  The sum
and product of two formal series are defined in the usual way.

\begin{lemma}\label{lem:natural-supports}
Let $A,B\subseteq\R$ be natural.
\begin{enumerate}
\item The union $A\cup B$ and the Minkowski sum $A+B$ are natural.
\item If $A\subseteq(0,\infty)$ is nonempty, then the additive monoid
$\langle A\rangle$ generated by $A$ is natural.
\end{enumerate}
\end{lemma}

\begin{proof}
The assertion for the union is immediate, and the empty cases for the
Minkowski sum are trivial.  Assume therefore that $A$ and $B$ are nonempty,
and put $a_0=\min A$ and $b_0=\min B$.  If $a+b\leq C$, then
$a\leq C-b_0$ and $b\leq C-a_0$; only finitely many pairs can occur.  Thus
$A+B$ is natural.

For the second statement, let $\delta=\min A>0$.  There is nothing to prove
when $C<0$.  A sum of elements of $A$ which is at most $C\geq0$ has length
at most $C/\delta$, and only elements of the finite set $A\cap(0,C]$ may
occur.  Hence there are only finitely many such sums.
\end{proof}

\begin{proposition}
\label{prop:natural-field}
With the usual formal operations, $\Gnat$ is a field and $\Pnat$ is a
subring.  Every nonzero $F\in\Gnat$ has a valuation
\[
  \val(F):=\min\supp(F)
\]
and a leading coefficient $\lc(F)$, and the sign of the leading coefficient
defines an ordering which makes $\Gnat$ an ordered field.
\end{proposition}

\begin{proof}
Lemma~\ref{lem:natural-supports} makes addition and multiplication
well-defined.  Let $F\neq0$, write $\lambda=\val(F)$ and
$a=\lc(F)$, and factor
$
  F=aX^\lambda(1+H),
$
where either $H=0$ or $\supp(H)\subseteq(0,\infty)$ is natural.  Formally,
\[
  F^{-1}=a^{-1}X^{-\lambda}\sum_{n=0}^{\infty}(-H)^n.
\]
The support of the geometric series is contained in the natural monoid
generated by $\supp(H)$.  If $H\neq0$ and
$\delta=\min\supp(H)>0$, then a fixed exponent $\gamma$ can occur in $H^n$
only when $n\delta\leq\gamma$, so only finitely many values of $n$ contribute;
for each such $n$, the Cauchy product has finite coefficient sums by
Lemma~\ref{lem:natural-supports}.  Thus the inverse lies in $\Gnat$.  The
assertions about the ordering are the standard leading-term verification.
\end{proof}

For $F=\sum_\alpha c_\alpha X^\alpha\in\Gnat$, define its formal derivative by
\[
  F':=\sum_\alpha \alpha c_\alpha X^{\alpha-1},
  \qquad
  \partial_XF:=XF'=\sum_\alpha \alpha c_\alpha X^\alpha.
\]
Translation of a natural support is natural, so $F'\in\Gnat$.  The usual
formal calculation shows that $F\mapsto F'$ is a derivation of $\Gnat$.
Thus $\Gnat$ is an ordered differential field.

\subsection{Convergence and asymptotics}

For a generalized series $F=\sum c_\alpha X^\alpha$ and $\rho>0$, put
$
  \norm{F}_\rho:=\sum_\alpha |c_\alpha|\rho^\alpha.
$
We say that $F$ \emph{converges at radius $\rho$} if this sum is finite.
The definition applies both to natural supports and to arbitrary
well-ordered supports.  A convergent series defines a real function for
$0<x\leq\rho$ by absolute convergence.

\begin{lemma}\label{lem:leading}
Let $
  F(X)=\sum_{\alpha\in A}c_\alpha X^\alpha$ 
be a nonzero generalized Laurent series which converges at some radius
$\rho>0$, where $A\subseteq\R$ is well ordered.  If
$\alpha_0=\min\supp(F)$, then
$
  F(x)=c_{\alpha_0}x^{\alpha_0}+o(x^{\alpha_0})
  \,\,\,\,\,\,(x\to0^+).
$
\end{lemma}

\begin{proof}
If the support consists only of $\alpha_0$, there is nothing to prove.
Otherwise let
$
  \alpha_1=\min\bigl(\supp(F)\setminus\{\alpha_0\}\bigr).
$
For $0<x<\rho$ we have
\begin{align*}
 \left|F(x)-c_{\alpha_0}x^{\alpha_0}\right|
 &\leq \sum_{\alpha\geq\alpha_1}|c_\alpha|x^\alpha \\
 &\leq (x/\rho)^{\alpha_1}
       \sum_{\alpha\geq\alpha_1}|c_\alpha|\rho^\alpha
 \leq \rho^{-\alpha_1}\norm{F}_\rho x^{\alpha_1}.
\end{align*}
Since $\alpha_1>\alpha_0$, the last quantity is
$o(x^{\alpha_0})$.
\end{proof}

\begin{definition}
Let $A\subseteq\R$ be natural and let
$F=\sum_{\alpha\in A}c_\alpha X^\alpha$.  A germ
$h:(0,\varepsilon)\to\R$ has \emph{natural-support asymptotic expansion}
$F$, written $h\sim F$, if for every $B\in\R$,
\[
  h(x)-\sum_{\substack{\alpha\in A\\ \alpha\leq B}}
  c_\alpha x^\alpha=o(x^B)
  \qquad(x\to0^+).
\]
The sum is finite by naturalness.
\end{definition}

\begin{lemma}\label{lem:formal-uniqueness}
A germ has at most one natural-support generalized Laurent asymptotic
expansion.
\end{lemma}

\begin{proof}
Suppose $F$ and $G$ are two such expansions and $F\neq G$.  The support of
$F-G$ is contained in the natural set $\supp(F)\cup\supp(G)$, so it has a
least exponent $\gamma$.  Let $P$ be the common finite sum of all terms of
exponent strictly below $\gamma$, and let $a$ and $b$ be the coefficients of
$X^\gamma$ in $F$ and $G$, respectively, taking a missing coefficient to be
zero.  The asymptotic relations at order $\gamma$ give
$
  h(x)-P(x)-ax^\gamma=o(x^\gamma),
$ and $
  h(x)-P(x)-bx^\gamma=o(x^\gamma).
$
Subtracting yields $(a-b)x^\gamma=o(x^\gamma)$, contrary to $a\neq b$.
\end{proof}

\begin{proposition}\label{prop:rigidity} 
Suppose a germ $h$ has a natural-support asymptotic expansion
$\widehat h\in\Gnat$.  If $h$ is represented by a convergent generalized
Laurent series $F$ with arbitrary well-ordered support, then $F=\widehat h$ 
as formal series.  In particular, $\widehat h$ converges.
\end{proposition}

\begin{proof}
Assume $F\neq\widehat h$.  The union of two well-ordered subsets of $\R$ is
well ordered: every nonempty subset meets one of the two sets, and the
minimum of the available minima is its least element.  Hence
$\supp(F-\widehat h)$ has a least exponent $\gamma$.

All common exponents below $\gamma$ belong to the natural support of
$\widehat h$, and therefore there are only finitely many of them.  Subtract
the corresponding common finite sum from both the germ and the convergent
series $F$.  If $\gamma\notin\supp(\widehat h)$, the defining asymptotic
relation at order $\gamma$ says that the resulting germ is $o(x^\gamma)$,
whereas Lemma~\ref{lem:leading} gives a nonzero leading term of order
$x^\gamma$ from $F$.  If $\gamma\in\supp(\widehat h)$, the same comparison
shows that the coefficients at $\gamma$ must agree, again contradicting the
choice of $\gamma$.  Thus $F=\widehat h$.
\end{proof}

\subsection{Two closure lemmas for convergent natural series}

The following elementary closure properties will later control analytic
changes of transversal coordinates.

\begin{lemma}\label{lem:functional-calculus}
Let $H\in\Pnat$ be convergent and satisfy either $H=0$ or
$\val(H)>0$.  If $A(T)=\sum_{n\geq0}a_nT^n$ is an ordinary real power
series convergent near $0$, then the formal composition $
  A(H)=\sum_{n\geq0}a_nH^n$ 
is a convergent natural-support generalized power series.
\end{lemma}

\begin{proof}
The assertion is immediate when $H=0$.  Otherwise choose $R>0$ with
$\norm{H}_R<\infty$ and put $\delta=\val(H)>0$.  For $0<r<R$,
\[
  \norm{H}_r
  \leq (r/R)^\delta\norm{H}_R\longrightarrow0
  \qquad(r\to0^+).
\]
Choose $r$ so small that $\norm{H}_r$ lies inside the disk of absolute
convergence of $A$.  Since $\val(H)>0$, only finitely many powers $H^n$ contribute to any
fixed coefficient, so the formal composition is well defined.  Its support
is contained in the additive monoid generated by $\supp(H)$, together
with $0$, and is natural by Lemma~\ref{lem:natural-supports}.  Moreover,
$
  \norm{A(H)}_r
  \leq\sum_{n\geq0}|a_n|\norm{H}_r^n<\infty.
$
Absolute convergence also justifies the collection of equal exponents.
\end{proof}

\begin{lemma}\label{lem:analytic-substitution}
Let $F\in\Pnat$ be convergent and let
\[
  q(X)=aX(1+u(X)),\qquad a>0,
\]
where $u$ is an ordinary convergent power series with $u(0)=0$.  Then
$F(q(X))$ is a convergent natural-support generalized power series.  The
same conclusion, with generalized Laurent series in place of power series,
holds after factoring a leading monomial.
\end{lemma}

\begin{proof}
Suppose $\norm{F}_\rho<\infty$.  For an ordinary power series we use the
same coefficient norm $\norm{\cdot}_r$.  Choose $r>0$ so small that
$s:=\norm{u}_r<1$.  Then
\[
  V:=\log(1+u)=\sum_{n\geq1}\frac{(-1)^{n+1}}n u^n
\]
converges in coefficient norm and
$
  \norm{V}_r\leq -\log(1-s)=:L.
$
For every $\alpha\geq0$,
$
  (1+u)^\alpha=\exp(\alpha V)
$
is an ordinary convergent power series and
$
  \norm{(1+u)^\alpha}_r\leq e^{\alpha L}.
$
After decreasing $r$ further, we may assume $are^L<\rho$.  If
$F=\sum_{\alpha\in A}c_\alpha X^\alpha$, then
\begin{align*}
 \sum_{\alpha\in A}
 \norm{c_\alpha q(X)^\alpha}_r
 &\leq
 \sum_{\alpha\in A}|c_\alpha|(are^L)^\alpha
 <\infty.
\end{align*}
The resulting support is contained in $A+\N$, which is natural by
Lemma~\ref{lem:natural-supports}.  This proves the power-series assertion.

For the Laurent case, assume $F\neq0$ and write
$F=X^\lambda P$ with $\lambda=\val(F)$ and $P\in\Pnat$.  Then
\[
  F(q(X))=a^\lambda X^\lambda(1+u(X))^\lambda P(q(X)).
\]
The preceding argument applies to $P(q(X))$.  Moreover, the ordinary series
\[
  (1+u)^\lambda=\exp\bigl(\lambda\log(1+u)\bigr)
\]
converges in coefficient norm at a sufficiently small radius, also when
$\lambda<0$.  The product is therefore a convergent natural-support
generalized Laurent series.
\end{proof}


\section{\texorpdfstring{The canonical expansion map for $\RQ$-germs}{The canonical expansion map for R-Q-germs}}
\label{sec:canonical}

We recall the one-variable facts from~\cite{KRS} that are used below.  A
one-variable germ $g\in\mathcal Q$ has a holomorphic representative on a
quadratic domain of the Riemann surface of the logarithm and a
natural-support generalized asymptotic expansion $Tg\in\Pnat$.  The map
\[
  T:\mathcal Q\longrightarrow\C[[X^{\R_{\geq0}}]]_{\mathrm{nat}}
\]
is an injective $\C$-algebra homomorphism, and $g(0)=(Tg)(0)$
\cite[Proposition~5.11]{KRS}.  Moreover, $g$ is a unit of the
$\mathcal Q$-algebra exactly when $g(0)\neq0$
\cite[Proposition~5.12]{KRS}.  If
$\partial g:=xg'(x)$ denotes logarithmic differentiation, then
$\partial g\in\mathcal Q$ and
$
  T(\partial g)=\partial_X(Tg)
$
by~\cite[Proposition~5.8]{KRS}.  Every monomial $x^\alpha$ with
$\alpha\geq0$ belongs to the one-variable $\mathcal Q$-class and has
expansion $X^\alpha$.  For a germ that is real valued on the positive axis,
Proposition~7.3 of~\cite{KRS} shows that the coefficients of $Tg$ are real.
We therefore regard $Tg$ as an element of $\Pnat$ for the real germs
considered below.

The unary preparation theorem of Kaiser--Rolin--Speissegger states that every
nonzero unary $\RQ$-definable germ can be written
\begin{equation}
\label{eq:KRS-preparation}
  f(x)=x^r g(x),
  \qquad r\in\R,\quad g\in\mathcal Q,\quad g(0)\neq0.
\end{equation}
See~\cite[Theorem~B]{KRS}.

\begin{lemma}\label{lem:unique-preparation}
If
$
  x^r g(x)=x^s h(x)
$
near $0^+$, where $g,h\in\mathcal Q$ and $g(0)h(0)\neq0$, then
$r=s$ and $g=h$.
\end{lemma}

\begin{proof}
The quotient $h(x)/g(x)$ tends to the finite nonzero limit $h(0)/g(0)$,
whereas it equals $x^{r-s}$.  This forces $r=s$, and then $g=h$.
\end{proof}

Let $\Hfield$ be the field of unary $\RQ$-definable germs at $0^+$.  It is
ordered by eventual positivity.  By unary preparation, every such germ is
real analytic on a punctured interval.  Its derivative is definable, so
$\Hfield$ is closed under ordinary differentiation.

\begin{definition}
For $f\in\Hfield\setminus\{0\}$, let $f=x^rg$ be the unique preparation
\eqref{eq:KRS-preparation} and define
$
  \Tcan(f):=X^rTg\in\Gnat.
$
Set $\Tcan(0)=0$.
\end{definition}

\begin{remark}[Source of the construction]
The map $T$ on the one-variable $\mathcal Q$-class and the preparation
\eqref{eq:KRS-preparation} are results of~\cite{KRS}.  The point of the
present formulation is that uniqueness of the prepared exponent makes their
combination choice-free on the full unary definable germ field; the next
theorem records the resulting field, order, differential, asymptotic, and
uniqueness properties.
\end{remark}

\begin{theorem}\label{thm:canonical-map}
The map
$
  \Tcan:\Hfield\longrightarrow\Gnat
$ 
is an injective ordered differential-field homomorphism.  For every $f\in\Hfield$,
$\Tcan(f)$ is the unique natural-support generalized Laurent asymptotic
expansion of $f$.
\end{theorem}

\begin{proof}
The definition is independent of choices by
Lemma~\ref{lem:unique-preparation}.  The asymptotic properties of $Tg$ in the
$\mathcal Q$-class show that $X^rTg$ is a natural-support generalized Laurent
asymptotic expansion of $x^rg=f$.  Its uniqueness follows from
Lemma~\ref{lem:formal-uniqueness}.

Let $f=x^rg$ and $h=x^sk$ be nonzero prepared germs.  Since $gk$ is a
$\mathcal Q$-unit and $T$ is an algebra homomorphism,
\[
  \Tcan(fh)=X^{r+s}T(gk)
            =(X^rTg)(X^sTk)=\Tcan(f)\Tcan(h).
\]
Likewise, $g^{-1}\in\mathcal Q$ and
$T(g^{-1})=(Tg)^{-1}$, so
$\Tcan(f^{-1})=\Tcan(f)^{-1}$.

For addition, assume $r\leq s$.  If $r<s$, then
$
  f+h=x^r\bigl(g+x^{s-r}k\bigr),
$
and the factor in parentheses is a $\mathcal Q$-unit.  Therefore
\[
  \Tcan(f+h)=X^r\bigl(Tg+X^{s-r}Tk\bigr)
             =\Tcan(f)+\Tcan(h).
\]
If $r=s$, then $f+h=x^r(g+k)$.  If $g+k=0$, the assertion is immediate.
Otherwise apply the unary preparation theorem to the $\RQ$-definable germ
$g+k$, say $g+k=x^tu$ with $u\in\mathcal Q$ a unit.  Since $g+k$ is bounded
at zero, $t\geq0$, and the monomial $x^t$ belongs to $\mathcal Q$.  The
injective algebra homomorphism $T$ gives
\[
  Tg+Tk=T(g+k)=X^tTu.
\]
It follows that
\[
  \Tcan(f+h)=X^{r+t}Tu=X^r(Tg+Tk)
             =\Tcan(f)+\Tcan(h).
\]
Thus $\Tcan$ is a field homomorphism.

If $f\neq0$, then the leading term of $\Tcan(f)=X^rTg$ is
$g(0)X^r$, which is nonzero; hence $\Tcan$ is injective.  Also,
$f(x)=x^r(g(0)+o(1))$, so the eventual sign of $f$ is the sign of the
leading coefficient $g(0)$.  Thus $\Tcan$ preserves the order.

It remains to verify compatibility with ordinary differentiation.  Let
$f=x^rg$ and put $\partial g=xg'$.  Then
$
  f'=x^{r-1}h$
  where $h:=rg+\partial g\in\mathcal Q,
$
and Proposition~5.8 of~\cite{KRS} gives
$
  Th=rTg+\partial_X(Tg).
$
If $h=0$, both $\Tcan(f')$ and $(\Tcan f)'$ vanish.  If $h\neq0$, apply the
unary preparation theorem to $h$: since $h$ is bounded at zero, there are
$t\geq0$ and a $\mathcal Q$-unit $u$ such that $h=x^tu$.  Hence
\[
  \Tcan(f')=X^{r-1+t}Tu
  =X^{r-1}Th
  =X^{r-1}\bigl(rTg+\partial_X(Tg)\bigr)
  =(X^rTg)'.
\]
Therefore $\Tcan(f')=(\Tcan f)'$, completing the proof.
\end{proof}

\begin{corollary}\label{cor:bounded-formal}
Every bounded unary $\RQ$-definable germ has a unique formal
natural-support generalized power-series expansion with nonnegative
exponents.
\end{corollary}

\begin{proof}
For a nonzero bounded germ $f=x^rg$ with $g(0)\neq0$, boundedness forces
$r\geq0$.  Hence $\Tcan(f)=X^rTg$ belongs to $\Pnat$.
\end{proof}

\begin{remark}
\label{rem:Tcan-on-Q}
If $f\in\mathcal Q$, then $\Tcan(f)=Tf$.  Indeed, both series are natural
asymptotic expansions of the same germ, so this follows from
Lemma~\ref{lem:formal-uniqueness}.
\end{remark}



\section{Convergence criterion and rigidity}
\label{sec:criterion}

We first identify convergent natural-support series inside the
one-variable $\mathcal Q$-class.

\begin{lemma}\label{lem:convergent-in-Q}
Let
$
  F(X)=\sum_{\alpha\in A}c_\alpha X^\alpha\in\Pnat
$
converge at some positive radius.  On each sufficiently small standard
quadratic domain, its holomorphic sum belongs to the one-variable class
$\mathcal Q$, and its $\mathcal Q$-expansion is $F$.
\end{lemma}

\begin{proof}
Suppose $\norm{F}_\rho<\infty$.  On the Riemann surface of the logarithm put
\[
  f(z):=\sum_{\alpha\in A}c_\alpha z^\alpha,
  \qquad z^\alpha:=\exp(\alpha\log z).
\]
The series converges normally on compact subsets of the logarithmic disk
$0<|z|<\rho$, since $|z^\alpha|=|z|^\alpha$.  It therefore defines a
holomorphic function there.  Fix a standard quadratic domain
$W_{c,C}$ with $0<c<\rho$.

Let $B\geq0$.  Naturalness makes $A\cap[0,B]$ finite.  If the remaining tail
is nonempty, let $\beta$ be its least exponent; then $\beta>B$, and for
$z\in W_{c,C}$ we have
\[
 \left|f(z)-\sum_{\alpha\leq B}c_\alpha z^\alpha\right|
 \leq \left(\frac{|z|}{\rho}\right)^\beta\norm{F}_\rho
 =O(|z|^\beta)=o(|z|^B).
\]
The estimate is uniform on every smaller standard quadratic domain.  If the
tail is empty, the assertion is immediate.  Thus $f$ belongs to the
one-variable asymptotic class $\mathcal A_1^1(W_{c,C})$ of~\cite{KRS} with
expansion $F$.  By~\cite[Remark~5.2(2),(3)]{KRS}, restriction to smaller
quadratic domains is harmless and
$\mathcal A_1^1(W_{c,C})=\mathcal Q_1^1(W_{c,C})$.  Hence $f\in\mathcal Q$
and its expansion is $F$.
\end{proof}

\begin{theorem}\label{thm:convergence-criterion}
Let $f\in\Hfield$.  The following are equivalent.
\begin{enumerate}
\item The germ $f$ is represented near $0^+$ by a convergent generalized
Laurent series with well-ordered support.
\item The canonical series $\Tcan(f)$ converges at some positive radius.
\end{enumerate}
Whenever these conditions hold, every convergent generalized Laurent-series
representation of $f$ is equal to $\Tcan(f)$.
\end{theorem}

\begin{proof}
If $f$ has a convergent representation, Proposition~\ref{prop:rigidity}
applied to the canonical natural asymptotic expansion
$\Tcan(f)$ shows that the representing series equals $\Tcan(f)$.

Conversely, suppose $f\neq0$, write $f=x^rg$ as in
\eqref{eq:KRS-preparation}, and assume
$\Tcan(f)=X^rTg$ converges at a radius $\rho>0$.  Then
\[
  \norm{Tg}_\rho=\rho^{-r}\norm{\Tcan(f)}_\rho<\infty,
\]
so the power series $Tg$ converges.  Let $G$ be its holomorphic sum.  By Lemma~\ref{lem:convergent-in-Q},
$G\in\mathcal Q$ and $TG=Tg$.  Injectivity of the $\mathcal Q$-expansion map
gives $G=g$.  Hence
$
  f(x)=x^rG(x)
$ 
is a convergent generalized Laurent-series representation.  The zero germ
is immediate, and uniqueness was proved in the first paragraph.
\end{proof}

\begin{corollary}
\label{cor:bounded-criterion}
For a bounded unary $\RQ$-definable germ, the following are equivalent:
being represented by a convergent generalized power series and convergence
of its canonical series.  No assumption of naturalness is needed for the
support of a competing convergent representation.
\end{corollary}

\begin{proposition}\label{prop:convergent-locus}
The set
\[
  \Hconv=\{f\in\Hfield:\Tcan(f)\text{ converges at some positive radius}\}
\]
is an ordered differential subfield of $\Hfield$.  It consists exactly of
the germs represented by convergent generalized Laurent series.
\end{proposition}

\begin{proof}
The last assertion is Theorem~\ref{thm:convergence-criterion}.  Let $F$ and
$G$ be convergent natural-support generalized Laurent series.  After passing
to a common smaller radius, $F+G$ and $FG$ converge by the triangle
inequality and the submultiplicativity of the coefficient norm.

If $F\neq0$, write
\[
  F=aX^\lambda(1+H),
  \qquad a\neq0,
\]
where $H=0$ or $\val(H)>0$.  If $H\neq0$, choose $R>0$ with
$\norm{H}_R<\infty$ and put $\delta=\val(H)$.  Then, for $0<r<R$,
$\norm{H}_r\leq(r/R)^\delta\norm{H}_R$.  Thus
$\norm{H}_r<1$ for all sufficiently small $r$, and the geometric series
for $(1+H)^{-1}$ converges in coefficient norm.  Hence $F^{-1}$ is
convergent.

Finally, suppose $\norm{F}_R<\infty$ and fix $0<r<R$.  Put $q=r/R$.  The
support of $F$ contains only finitely many negative exponents, while
$\sup_{\alpha\geq0}\alpha q^\alpha<\infty$.  Consequently,
\[
  \norm{F'}_r
  =r^{-1}\sum_\alpha |c_\alpha|R^\alpha |\alpha|q^\alpha
  <\infty.
\]
The differential-field property of $\Tcan$ now shows that $\Hconv$ is closed
under differentiation.  The ordering is inherited from $\Hfield$.
\end{proof}

This proves all assertions of Theorem~\ref{thm:A} except the properness of
$\Hconv$, which will follow from the divergent transition germ in
Corollary~\ref{cor:proper-convergent-locus}.

%


\section{\texorpdfstring{Principal branches and branching of $\RQ$}{Principal branches and branching of R-Q}}
\label{sec:branching}

Let
$
  \Lsurf=\{(r,\theta):r>0,\ \theta\in\R\}
$
be the Riemann surface of the logarithm, with projection
$(r,\theta)\mapsto re^{i\theta}$.  A standard quadratic domain is a set
\[
  W_{c,C}=\bigl\{(r,\theta)\in\Lsurf:
       0<r<c\exp(-C\sqrt{|\theta|})\bigr\},
  \qquad c,C>0.
\]
Every quadratic domain contains a standard one.  After decreasing
$\delta>0$, the standard domain $W_{c,C}$ contains
\begin{equation}
\label{eq:angular-strip}
  \{(r,\theta):0<r<\delta,\ |\theta|\leq\pi\},
\end{equation}
since one may take $\delta<c\exp(-C\sqrt\pi)$.

The next lemma is the only technical point in the branching argument.  Its
proof is local in the angular variable.  Kaiser's regular blow-up turns a
small angular sector into an ordinary analytic variable; finitely many such
sectors cover the principal branch.  Reflection across the positive axis
then gives the lower half-disk.

\begin{lemma}\label{lem:principal}
Let $U\subseteq\Lsurf$ be a quadratic domain and let
$\Phi\in\mathcal Q(U)$.  Assume that $\Phi$ is real valued on the positive
real axis.  Then, for some $\delta>0$, the restriction of $\Phi$ through
the principal branch of the logarithm to
\[
  \Delta_\delta=D(0,\delta)\setminus(-\delta,0]
\]
is $\RQ$-definable as a complex-valued function of two real variables.
\end{lemma}

\begin{proof}
By~\cite[Remark~5.2(2)]{KRS}, we may restrict $\Phi$ to a symmetric
standard quadratic domain $W_{c,C}\subseteq U$.  After decreasing a radial
bound, $W_{c,C}$ contains all $(r,\theta)$ with
$0<r<\delta_0$ and $|\theta|\leq\pi$.

We first treat the upper half-disk.  We follow Case~3 in the proof of
\cite[Theorem~3.3]{Kaiser}.  Introduce a dummy variable by
$
  f(z_1,z_2)=\Phi(z_2).
$
The dummy-variable and permutation properties
\cite[Remark~5.2(4) and Lemma~5.5(2)]{KRS} put $f$ in the two-variable
$\mathcal Q$-class used in Kaiser's argument.  For $a\in[0,\pi]$, put
$\lambda_a=e^{ia}$.  The regular blow-up construction
\cite[Remark~3.2]{Kaiser}, based on
\cite[Propositions~4.4 and~5.15]{KRS}, gives
$
  g_a:=\mathbf r^{1,\lambda_a}f
$
in the corresponding mixed class.  Define
\[
  h_a(z)=e^{i(z+a)}-e^{ia},
  \qquad
  G_a(z_1,z_2)=g_a(z_1,h_a(z_2)).
\]
Analytic composition preserves the mixed class by
\cite[Proposition~5.10]{KRS}.

After shrinking to conjugation-invariant domains, write
$z_1^*=(r,-\theta)$ for $z_1=(r,\theta)$ and set
\[
  G_a^*(z_1,z_2)=\overline{G_a(z_1^*,\overline{z_2})}.
\]
The conjugation argument in Case~3 of~\cite[Theorem~3.3]{Kaiser}, using
\cite[Proposition~7.3]{KRS}, keeps $G_a^*$ in the same mixed class.  Put
\[
  RG_a=\frac{G_a+G_a^*}{2},
  \qquad
  JG_a=\frac{G_a-G_a^*}{2i}.
\]
The mixed-class construction in
\cite[Definition~7.1, Proposition~7.3, and the discussion before
Theorem~7.9]{KRS} yields $\varepsilon_a>0$ such that the restrictions of
$RG_a$ and $JG_a$ to
$
  [0,\varepsilon_a]\times[-\varepsilon_a,\varepsilon_a]
$
are $\RQ$-definable.  On real arguments they are
\[
  (r,\varphi)\longmapsto
  \operatorname{Re}\Phi(re^{i(a+\varphi)}),
  \qquad
  (r,\varphi)\longmapsto
  \operatorname{Im}\Phi(re^{i(a+\varphi)}).
\]
This is the computation in Kaiser's Case~3.  Finitely many intervals
$(a-\varepsilon_a,a+\varepsilon_a)$ cover $[0,\pi]$.  Taking a common
smaller radial bound and using the definable polar-coordinate map gives
definability of $\Phi$ on a sufficiently small upper half-disk.

For the lower half-disk, define
$
  \Phi^*(r,\theta):=\overline{\Phi(r,-\theta)}.
$
The functions $\Phi$ and $\Phi^*$ are holomorphic on the connected domain
$W_{c,C}$ and agree on a positive real interval.  The identity theorem
therefore gives
\[
  \Phi(r,-\theta)=\overline{\Phi(r,\theta)}.
\]
Complex conjugation is semialgebraic, so reflection of the definable
upper-half-disk restriction gives a definable lower-half-disk restriction.
The two restrictions agree on the positive axis and define $\Phi$ on a
smaller principal slit disk.
\end{proof}

\begin{theorem}
\label{thm:RQ-branching}
The o-minimal structure $\RQ$ is branching.
\end{theorem}

\begin{proof}
Let $f:(0,\varepsilon)\to\R$ be $\RQ$-definable.  The zero germ is
immediate.  Otherwise, the unary preparation theorem
\cite[Theorem~B]{KRS} gives, after shrinking the interval,
\[
  f(x)=x^rg(x),
  \qquad r\in\R,\quad g\in\mathcal Q,\quad g(0)\neq0.
\]
Choose a complex representative $\Phi\in\mathcal Q(U)$ of $g$.  It is real
on the positive axis, and Lemma~\ref{lem:principal} makes its restriction
to a small principal slit disk both holomorphic and $\RQ$-definable.

The structure $\RQ$ contains $\Ran$ and defines every real power function
\cite[Introduction]{KRS}.  On a principal slit disk,
\[
  z^r=|z|^r\bigl(\cos(r\arg z)+i\sin(r\arg z)\bigr).
\]
The modulus is semialgebraic.  The principal argument is definable on a
finite semialgebraic partition by restricted arctangent.  Since
$\arg z\in(-\pi,\pi)$, the sine and cosine are evaluated on bounded
intervals and are restricted analytic functions.  Thus the principal
branch of $z^r$ is $\RQ$-definable.  The product $z^r\Phi(z)$ is the
required definable holomorphic extension of $f$.
\end{proof}



\section{Dulac convergence and analytic coordinates}\label{sec:coordinates}
For a positive germ $f(x)\to0$ as $x\to0^+$, put
\[
  \zeta=-\log x,
  \qquad
  D_f(\zeta):=-\log f(e^{-\zeta}).
\]
We use right half-planes; replacing $\zeta$ by $-\zeta$ gives the usual
left-half-plane convention.

\begin{definition}
\label{def:conv-dulac}
A \emph{convergent log-free Dulac representation} of $D_f$ is an identity
\begin{equation}
\label{eq:conv-dulac}
  D_f(\zeta)=\lambda\zeta+c+L(e^{-\zeta})
\end{equation}
for all sufficiently large real $\zeta$, where $\lambda>0$, $c\in\R$, and
$L\in\Pnat$ is convergent with $L(0)=0$.
\end{definition}

If $L(X)=\sum c_\alpha X^\alpha$ converges at a radius $\rho>0$, then
$L(e^{-\zeta})=\sum c_\alpha e^{-\alpha\zeta}$ converges normally on every
closed half-plane $\operatorname{Re}\zeta\geq R$ with $e^{-R}<\rho$.  Thus
the right-hand side of~\eqref{eq:conv-dulac} is holomorphic on a far right
half-plane.

\begin{proposition}\label{prop:power-dulac-equivalence}
Let $f:(0,\varepsilon)\to(0,\infty)$ tend to $0$.  The following are
equivalent.
\begin{enumerate}
\item The germ $f$ is represented by a convergent natural-support
generalized power series with positive leading coefficient and positive
leading exponent.
\item The logarithmic germ $D_f$ has a convergent log-free Dulac
representation.
\end{enumerate}
\end{proposition}

\begin{proof}
Suppose first that
\[
  f(x)=F(x),
  \qquad
  F(X)=aX^\lambda(1+H(X)),
\]
where $a>0$, $\lambda>0$, and either $H=0$ or $\val(H)>0$.  If
$H\neq0$, choose a small radius $r$ with $\norm{H}_r<1$.  By
Lemma~\ref{lem:functional-calculus},
\[
  -\log(1+H)=\sum_{n\geq1}\frac{(-1)^n}{n}H^n
\]
is a convergent element of $\Pnat$; the case $H=0$ is immediate.  Hence,
for all sufficiently large real $\zeta$,
\[
  D_f(\zeta)
  =\lambda\zeta-\log a-\log(1+H(e^{-\zeta})),
\]
which has the form~\eqref{eq:conv-dulac}.

Conversely, suppose
$D_f(\zeta)=\lambda\zeta+c+L(e^{-\zeta})$.  The functional calculus gives
a convergent natural-support power series $\exp(-L)$ with constant term
$1$.  Therefore
\[
  f(x)=\exp(-D_f(-\log x))
      =e^{-c}x^\lambda\exp(-L(x))
\]
is represented by a convergent natural-support generalized power series.
\end{proof}

The next result shows that this property does not depend on analytic
coordinates on the transversals.

\begin{theorem}\label{thm:coordinate-invariance}
Let $d:(0,\varepsilon)\to(0,\varepsilon')$ be positive and tend to $0$.
Let $\phi$ and $\psi$ be orientation-preserving real-analytic
diffeomorphism germs at $0$, fixing $0$, and put
$
  \widetilde d=\psi\circ d\circ\phi^{-1}.
$
Then the following are equivalent.
\begin{enumerate}
\item $d$ has a convergent natural-support generalized power-series
representation.
\item $\widetilde d$ has a convergent natural-support generalized
power-series representation.
\item $D_d$ has a convergent log-free Dulac representation.
\item $D_{\widetilde d}$ has a convergent log-free Dulac representation.
\end{enumerate}
\end{theorem}

\begin{proof}
Write
\[
  \phi^{-1}(x)=ax(1+u(x)),
  \qquad a>0,\quad u(0)=0,
\]
and
\[
  \psi(y)=by(1+v(y)),
  \qquad b>0,\quad v(0)=0.
\]
Suppose that $d$ is represented by a convergent natural-support series
$F$ of positive valuation.  Lemma~\ref{lem:analytic-substitution} shows
that $F\circ\phi^{-1}$ is convergent.  It still has positive valuation, so
Lemma~\ref{lem:functional-calculus}, applied to the ordinary analytic
series $v$, gives
\[
  \psi(F\circ\phi^{-1})
  =b(F\circ\phi^{-1})\bigl(1+v(F\circ\phi^{-1})\bigr),
\]
again a convergent natural-support generalized power series.  This proves
(1)$\Rightarrow$(2).

For the reverse implication, write
$
  d=\psi^{-1}\circ\widetilde d\circ\phi.
$
Both $\phi$ and $\psi^{-1}$ have the same analytic form used above, so the
same substitution argument applies directly to this expression.  Thus
(1) and (2) are equivalent.  Proposition
\ref{prop:power-dulac-equivalence} gives their equivalence with (3) and
(4).
\end{proof}

Equivalently, if
$
  \Phi(\zeta)=-\log\phi(e^{-\zeta}),
$ and $
  \Psi(\eta)=-\log\psi(e^{-\eta}),
$
then
\[
  D_{\widetilde d}=\Psi\circ D_d\circ\Phi^{-1}.
\]
The theorem says that convergence of a log-free Dulac representation is preserved by these analytic conjugacies.





\section{A divergent nonresonant saddle transition map}
\label{sec:saddles}

Let $p$ be a hyperbolic saddle of a real-analytic planar vector field.
Choose oriented analytic incoming and outgoing transversals and analytic
coordinates on them.  The local transition map
$
  d:(0,\varepsilon)\longrightarrow(0,\varepsilon')
$
is positive and tends to $0$.  Dulac's asymptotic expansion has the form
\begin{equation}
\label{eq:dulac-general}
  d(x)\sim p_0x^{\nu_0}+
  \sum_{j=1}^{\infty}p_j(\log x)x^{\nu_j},
  \qquad
  0<\nu_0<\nu_1<\cdots,\quad \nu_j\to\infty,
\end{equation}
where $p_0>0$ and the $p_j$ are real polynomials; see
\cite[pp.~2--3]{KRS}.  When the eigenvalue ratio is irrational, the saddle
is nonresonant and all $p_j$ are constant.  In that case
\eqref{eq:dulac-general} is a natural-support generalized power series.

We first recall the classical source of divergence.  The distinction between
analytic equivalence and orbital analytic equivalence is important here,
because transition maps depend only on the orbit foliation.

\begin{theorem}\label{thm:classical-input}
There is a real-analytic nonresonant hyperbolic saddle whose local
transition map has a divergent Dulac series.  This conclusion is independent
of the analytic coordinates chosen on the incoming and outgoing
transversals.
\end{theorem}

\begin{proof}[Classical source]
Trifonov recalls Ilyashenko's small-divisor examples
\cite{Ilyashenko81} and explains that, for irrational eigenvalue ratios that
are too well approximated by rationals, Main Local Theorem~1 yields divergent
Dulac series~\cite[Introduction and Main Local Theorem~1]{Trifonov}.  That
theorem characterizes convergence of the loop monodromy series by orbital
analytic equivalence of the saddle germ to its formal normal form.
Trifonov's realization theorem places any analytic saddle germ, up to orbital
analytic equivalence, at the vertex of a real-analytic separatrix loop
\cite[Theorem~2]{Trifonov}.  The localization argument shows that the Dulac
series of the local saddle transition and the loop monodromy converge or
diverge together, and that this property is independent of the analytic
transversal coordinates~\cite[Section~1.1]{Trifonov}.  The nonresonant case
is treated explicitly in~\cite[Section~1.7(a)]{Trifonov}.  This proves the
theorem.
\end{proof}

Fix a saddle as in Theorem~\ref{thm:classical-input}.  Choose analytic
transversal coordinates satisfying condition~(D) of
Kaiser--Rolin--Speissegger~\cite{KRS}, and let $d$ be the corresponding
local transition germ.  The coordinate-independence in
Theorem~\ref{thm:classical-input} shows that its Dulac series is still
divergent.

Kaiser--Rolin--Speissegger show that $d$ belongs to the one-variable
$\mathcal Q$-class and that its restriction to every sufficiently small
positive interval is $\RQ$-definable; see
\cite[pp.~2--4 and the corollary following Theorem~B]{KRS}.  Its
$\mathcal Q$-expansion is the nonresonant Dulac series
\begin{equation}
\label{eq:Td}
  Td=\sum_{n=0}^{\infty}a_nX^{\mu_n},
  \qquad
  a_0>0,\quad
  0<\mu_0<\mu_1<\cdots,\quad \mu_n\to\infty.
\end{equation}
By Remark~\ref{rem:Tcan-on-Q}, this is the canonical expansion
$\Tcan(d)$.

\begin{proposition}
\label{prop:divergent-transition}
The canonical generalized power series $\Tcan(d)=Td$ is divergent.  The
germ $d$ has no convergent generalized Laurent-series representation with
well-ordered support.
\end{proposition}

\begin{proof}
Suppose that $Td$ converges at some radius.  By
Theorem~\ref{thm:convergence-criterion}, the germ $d$ is then represented by
$Td$.  Writing $Td=\sum_n a_nX^{\mu_n}$, coefficient-norm convergence
implies normal convergence of $\sum_n a_nz^{\mu_n}$ on compact subsets of
a logarithmic disk.  Its sum is therefore holomorphic there.  On the other
hand, the $\mathcal Q$-representative of $d$ is holomorphic on a standard
quadratic domain.  The two functions agree on a positive real interval, so
the identity theorem makes them agree on the connected component of the
overlap containing that interval.  Thus $Td$ is a convergent Dulac-series
representation of the transition map, contrary to
Theorem~\ref{thm:classical-input}.  Hence $Td$ diverges.  The final assertion
follows from Theorem~\ref{thm:convergence-criterion}.
\end{proof}

\begin{corollary}
\label{cor:proper-convergent-locus}
The ordered differential subfield $\Hconv$ is a proper subfield of
$\Hfield$.
\end{corollary}

\begin{proof}
The transition germ $d$ belongs to $\Hfield$ but not to $\Hconv$ by
Proposition~\ref{prop:divergent-transition}.
\end{proof}

Corollary~\ref{cor:proper-convergent-locus} completes the proof of
Theorem~\ref{thm:A}.  Theorem~\ref{thm:RQ-branching} and
Proposition~\ref{prop:divergent-transition} prove Theorem~\ref{thm:B}.


\section{Consequences}
\label{sec:consequences}

\subsection{Formal and convergent expansions}

Dembner asks in~\cite[Question~3.17]{Dembner} whether every unary germ
definable in a branching structure admits a generalized power-series
expansion.  He clarified to the author that ``expansion'' is intended to mean
a convergent generalized power-series representation.  The following
corollary gives a negative answer and records the contrasting formal result
for $\RQ$.

\begin{corollary}
\label{cor:dembner-answer}
In the branching structure $\RQ$, the following hold.
\begin{enumerate}
\item Every unary definable germ has a unique formal natural-support
generalized Laurent expansion.
\item Every bounded unary definable germ has a unique formal natural-support
generalized power-series expansion.
\item There is a positive bounded unary definable germ with no convergent
generalized power-series representation.
\end{enumerate}
Consequently, the answer to the question posed
in~\cite[Question~3.17]{Dembner} is negative.  By contrast, every bounded
unary germ in $\RQ$ has a canonical formal generalized power-series
expansion.
\end{corollary}

\begin{proof}
The first assertion follows from Theorem~\ref{thm:canonical-map}, the second
from Corollary~\ref{cor:bounded-formal}, and the third from
Proposition~\ref{prop:divergent-transition}.
\end{proof}

\begin{remark}
The formal conclusion is specific to $\RQ$.  It does not show that every
unary germ in an arbitrary branching structure, or even in an arbitrary
polynomially bounded o-minimal structure, has a formal generalized
asymptotic expansion.  The counterexample in part~(3) is bounded and tends
to zero, so the negative answer is independent of whether Laurent monomials
are allowed.
\end{remark}

\subsection{Comparison with convergent-series structures}

\begin{corollary}
\label{cor:strict-inclusion}
There is a strict reduct inclusion
$
  \RanR\subsetneq\RQ.
$
Moreover, the transition germ $d$ from
Proposition~\ref{prop:divergent-transition} is not definable in
$\Ranstar$.
\end{corollary}

\begin{proof}
The structure $\RQ$ contains $\Ran$ and defines every real power function
\cite[Introduction]{KRS}, so $\RanR$ is a reduct of $\RQ$.  Suppose that $d$ were definable in
$\RanR$.  By Miller's unary preparation theorem
\cite[Proposition~4.5]{Miller}, after shrinking one could write
\[
  d(x)=x^{r_0}F(x^{r_1},\ldots,x^{r_k}),
  \qquad r_i>0\quad(i\geq1),
\]
with $F$ an ordinary convergent power series.  Expanding $F$ gives
\[
  d(x)=\sum_{I\in\N^k}a_Ix^{r_0+r_1i_1+\cdots+r_ki_k}.
\]
Write $F(Y)=\sum_{I\in\N^k}a_IY^I$ and choose positive numbers
$R_1,\ldots,R_k$ in a polydisk of absolute convergence, so that
\[
  \sum_{I\in\N^k}|a_I|R_1^{i_1}\cdots R_k^{i_k}<\infty.
\]
For sufficiently small $x>0$ we have $x^{r_j}<R_j$ for all $j\geq1$, and
\[
 \sum_{I\in\N^k}|a_I|
 x^{r_0+r_1i_1+\cdots+r_ki_k}
 =x^{r_0}\sum_{I\in\N^k}|a_I|
   (x^{r_1})^{i_1}\cdots(x^{r_k})^{i_k}<\infty.
\]
Moreover, for every exponent bound $B$, the inequality
$r_0+r_1i_1+\cdots+r_ki_k\leq B$ permits only finitely many multi-indices,
because each $r_j>0$.  After collecting equal exponents, we therefore obtain
a convergent natural-support generalized Laurent series, contradicting
Theorem~\ref{thm:convergence-criterion} and
Proposition~\ref{prop:divergent-transition}.  Thus the reduct inclusion is
strict.

By the one-variable preparation theorem for $\Ranstar$
\cite[Theorem~B]{DS98} (see also the proof of
\cite[Proposition~3.14]{Dembner}), a unary definable germ is a monomial times
a convergent generalized power series with well-ordered support.  Applied
to $d$, this gives a convergent generalized Laurent series.  Since
$d(x)\to0$, its leading exponent is positive; after absorbing the leading
monomial into the exponents, one obtains a convergent generalized power
series.  Proposition~\ref{prop:divergent-transition} therefore shows that
$d$ is not $\Ranstar$-definable.
\end{proof}

\subsection{Definable complex curves}

The branching theorem also places $\RQ$ directly within the scope of
Dembner's geometric results.

\begin{corollary}
\label{cor:complex-curves}
The following statements hold for $\RQ$.
\begin{enumerate}
\item Let $Y$ be an irreducible compact complex curve and let
$X\subsetneq Y$ be a noncompact $\RQ$-definable open subset whose
normalization has no punctures.  Every definable coherent sheaf on $X$ has
vanishing higher cohomology, and every definable vector bundle on $X$ is
definably trivial.
\item Every $\RQ$-definable Riemann surface without definable punctures
admits a definable open embedding into a compact Riemann surface.  The
boundary may be chosen to be a finite union of simple closed curves,
with the surface connected at each boundary point.
\end{enumerate}
\end{corollary}

\begin{proof}
The structure $\RQ$ contains $\Ran$ and is branching by
Theorem~\ref{thm:RQ-branching}.  The two assertions are therefore immediate
from Dembner's Theorems~A and~C, respectively~\cite{Dembner}.
\end{proof}


\end{document}